\documentclass[12pt, leqno]{amsart}

\usepackage{amsfonts,amsthm,amsmath,xcolor}
\numberwithin{equation}{section}

\usepackage{amssymb}

\theoremstyle{plain}
\newtheorem{thm}{Theorem}[section]

\newtheorem{lem}{Lemma}[section]

\theoremstyle{definition}

\newcommand{\ZZ}{\mathbb{Z}}

\DeclareMathOperator{\cay}{Cay}

\begin{document}

%%%%%%%%%%%%%%%%%%%%%%%%%%%%%%%%%

\title{On the average hitting times of $\cay(\ZZ_N,\{+1,+2\})$}
\author{Yuuho Tanaka}
\address{Graduate School of Science and Engineering, Waseda University, Tokyo, 169-8555, Japan} \email{tanaka\_yuuho\_dc@akane.waseda.jp} 
%}
\date{}
\maketitle

\begin{abstract}
The exact formula for the average hitting time (HT, as an abbreviation) of simple random 
walks %from one vertex to any other vertex 
on $\cay(\ZZ_N,\{\pm1,\pm2\})$ % of an $N$-vertex cycle graph $\cay(\ZZ_N,\{\pm1\})$ 
was given by %N. Chair [\textit{Journal of Statistical Physics}, \textbf{154} (2014) 1177-1190] and 
Y. Doi et al. [\textit{Discrete Applied Mathematics}, \textbf{313} (2022) 18-28].
%N. Chair gives the expression for the even $N$ case and the expression for the odd $N$ case separately.
Y. Doi et al. give a simple formula for the HT's of simple random walks on $\cay(\ZZ_N,\{\pm1,\pm2\})$ by using an elementary method.
%different from Chair (2014).
In this paper, using an elementary method also used by Y. Doi et al. \cite{HT}, we give a simple formula for HT's of simple random walks on $\cay(\ZZ_N,\{+1,+2\})$. 
%Our proof is considerably short and fully combinatorial, in particular, has no-need of any spectral graph theoretical arguments. Not only the formula itself but also intermediate results through the process of our proof describe clear relations between the HT's of simplerandom walks on $\cay(\ZZ_n,\{+1,+2\})$ and the Jacobsthal numbers.
\end{abstract}

{\small
	\noindent
	{\bfseries Key Words:}
	simple random walk, hitting time, Cayley graph, Jacobsthal number%\\ \vspace{-0.15in}
	
	%\noindent
	%2010 {\it Mathematics Subject Classification}. 
	%Primary 11T71;
	%Secondary 94B05, 11F11.\\ \quad
}

\section{Introduction}

A simple random walk on a graph $G$ is a discrete stochastic model
such that a random walker at a vertex $u\in V(G)$ moves to a vertex $v$
adjacent to the vertex $u$ at the next step with the probability of
$\frac{1}{\deg_G(u)}$, where $\deg_G(u)$ denotes the degree of the vertex $u$
of $G$.
The number of steps required for the random walker starting at vertex $p$ of $G$ to reach vertex $q$ of $G$ for the first time is called {\it the hitting time} from $p$ to $q$ of the simple random walk on $G$.
{\it The average hitting time} (HT, as an abbreviation) from $p$ to $q$ on $G$, 
denoted by $h(G;p,q)$, means the expected value of the hitting times
from $p$ to $q$ of simple random walks on $G$.
Note that $h(G;p,p)=0$.
The exact formula for the average hitting time from one vertex to any other vertex is far from
available in general.
For some very special graph classes with high symmetries, it may be possible to obtain such exact formulas. 
%For example, as far as the authors know, only in the case of $i\in\{1,2\}$, exact formulas for the $i$-th power $C^i_N$ of an $N$-vertex cycle graph $C_N$ are obtained.

Let $G$ be a finite group and $S\subseteq G$ be a subset.
The corresponding {\it Cayley graph} $\cay({\it G,S})$ has vertex set equal to $G$.
Two vertices $g,h\in G$ are joind by a directed edge from $g$ to $h$ if and only if there exists $s\in S$ such that $g=sh$.
%It was not until 2014 that such a formula for the case of $i=2$ was given by N. Chair \cite{Chair}.
In N. Chair \cite{Chair}, by calculating the exact value of Wu's formula \cite{Wu} of the effective resistances in terms of the eigenvalues and the eigenvectors of the Laplacian matrices of $\cay(\ZZ_N,\{\pm1,\pm2\})$, the author gave the expression depending on the parity of $N$.
%Let us define $V(C^2_{N}):= \mathbb{Z}_{N}$ and $E(C^2_{N}):= \{ \{u,v\} \mid u-v=\pm1,\pm2 \}$. 
In Y. Doi et al. \cite{HT}, by using an elementary method, they gave a much simpler formula than Chair's formula \cite{Chair}.
This formula does not depend on the parity of $N$.
%Note that $h_N(G;0,0)=0$.
%Since the dihedral group $D_N$ acts on $\cay(\ZZ_N,\{\pm1,\pm2\})$, we have that $\forall k,l \in\mathbb{Z}_N,h_N(\cay(\ZZ_N,\{\pm1,\pm2\});0,l)=h_N(\cay(\ZZ_N,\{\pm1,\pm2\});k,k+l)=h_N(\cay(\ZZ_N,\{\pm1,\pm2\});k+l,k)$.

%\begin{thm}[Chair, 2014 \cite{Chair}]
%The exact formula for the HT's of simple random walks on $C^2_N(N\geq 5)$ is as follows:
%\begin{enumerate}
%\item For the case of $N=2n$,  
%\[
%h_N(0,l)=\frac{2}{5}l\left(N-l\right)+\left(-1\right)^{l+1}\frac{2N}{\sqrt{5}}F_l^2\left(\frac{1+\left(\frac{3-\sqrt{5}}{2}\right)^N}{1-\left(\frac{3-\sqrt{5}}{2}\right)^N}\right)+\left(-1\right)^l\frac{2N}{5}F_{2l}.
%\]
%\item For the case of $N=2n+1$, 
%\[
%h_N(0,l)=\frac{2}{5}l\left(N-l\right)+\left(-1\right)^{l+1}\frac{2N}{\sqrt{5}}F_l^2\left(\frac{1-\left(\frac{3-\sqrt{5}}{2}\right)^N}{1+(\frac{3-\sqrt{5}}{2})^N}\right)+\left(-1\right)^l\frac{2N}{5}F_{2l}.
%\]
%\end{enumerate}
%\end{thm}

\begin{thm}[Y. Doi et al., 2022 \cite{HT}]\label{main}
Let $F_i$ be the $i$-th Fibonacci number.
The exact formula for the HT's of simple random walks on $\cay(\ZZ_N,\{\pm1,\pm2\})$ $(N \geq 5)$ is, 
\[
h(\cay(\ZZ_N,\{\pm1,\pm2\});0,l)=\frac{2}{5}\left(l\left(N-l\right)+2N\frac{F_{l}F_{N-l}}{F_{N}}\right).
\] 
\end{thm}
Its proof is considerably short and fully combinatorial.
Furthermore, it does not require any spectral graph theoretical arguments.
%Not only the formula itself but also intermediate results through the process of
%our proof describe clear relations between the HT's of simple random walks on
%$C^2_N$ and the Fibonacci numbers. 

In this paper, using an elementary method also used by Y. Doi et al. \cite{HT}, we give a simple formula for HT's of simple random walks on $\cay(\ZZ_N,\{+1,+2\})$. 
The organization of this paper is as follows.
In Section 2, we fix the notation for the graphs $\cay(\ZZ_N,\{+1,+2\})$ and give some properties of the Jacobsthal sequence.
In Section 3, we give the entries of the inverse of $H_N$.
As a consequence, we obtain the exact formula for HT on $\cay(\ZZ_N,\{+1,+2\})$ by calculation.
This equation does not depend on the parity of $N$.

%%%%%%%%%%%%%%%%%%%%%%%%%%%%%%%%%%%%%%%%%%%%%%%%%%
\section{Preliminary}

The sequence of numbers $J_n$ defined by the recurrence relation $J_0=0$, $J_1=1$, $J_{n+2}=J_{n+1}+2J_n$ $(n=0,1,2,\ldots)$ is called the {\it Jacobsthal sequence}.
We prepare some well-known formulas on Jacobsthal numbers, 
which will be used throughout this paper. 
For proofs of these formulas, please refer to the appropriate
literature on Jacobsthal numbers (e.g., \cite{Horadam}).

\begin{align}
J_n&=\frac{1}{3}(2^n-(-1)^n) \label{j0}\\%一般項
J_{n+1}+J_n&=2^n \label{j1}\\
J_{n+1}-2J_n&=(-1)^n \label{j2}\\
%J_{n+1}+2J_{n-1}&=3J_n+2(-1)^n \label{j3}\\
J_m(J_{n+1}+2J_{n-1})+J_n(J_{m+1}+2J_{m-1})&=2J_{m+n} \label{j4}
%J_m(J_{n+1}+2J_{n-1})-J_n(J_{m+1}+2J_{m-1})&=(-1)^n2^{n+1}J_{m-n} \label{j5}
\end{align}

\begin{lem}
\begin{align}
J_{n-1}&=J_jJ_{n-j}+2J_{j-1}J_{n-j-1} \label{lj1}\\
\sum_{j=1}^n2^jJ_{n-j}&=\frac{2}{3}(nJ_{n-1}+(n-1)J_n). \label{lj2}\\
\sum_{j=1}^n(-1)^{n+j}J_j&=\frac{1}{3}(nJ_{n-1}-(n-2)J_n) \label{lj3}
\end{align}
\end{lem}

\begin{proof}
(\ref{lj1})
\begin{align*}
J_{n-1}&=\frac{1}{2}(J_{n-j}(J_j+2J_{j-2})+J_{j-1}(J_{n-j+1}+2J_{n-j-1})) &&\text{by (\ref{j4})}\\
&=\frac{1}{2}(J_{n-j}(2J_j-J_{j-1})+J_{j-1}(J_{n-j}+4J_{n-j-1}))\\
&=J_jJ_{n-j}+2J_{j-1}J_{n-j-1}.
\end{align*}

(\ref{lj2})
\begin{align*}
\sum_{j=1}^n2^jJ_{n-j}&=\frac{1}{3}\sum_{j=1}^n2^j(2^{n-j}-(-1)^{n-j}) &&\text{by (\ref{j0})}\\
&=\frac{1}{3}(2^n\sum_{j=1}^n1+2(-1)^n\sum_{j=1}^n(-2)^{j-1})\\
&=\frac{1}{3}(2^nn+2(-1)^n\frac{1-(-2)^n}{3})\\
&=\frac{1}{3}(2n(J_{n-1}+J_n)-\frac{2}{3}(2^n-(-1)^n)) &&\text{by (\ref{j1})}\\
&=\frac{1}{3}(2n(J_{n-1}+J_n)-2J_n) &&\text{by (\ref{j0})}\\
&=\frac{2}{3}(nJ_{n-1}+(n-1)J_n).
\end{align*}

(\ref{lj3})
\begin{align*}
\sum_{j=1}^n(-1)^{n+j}J_j&=\frac{1}{3}\sum_{j=1}^n(-1)^{n+j}(2^j-(-1)^j) &&\text{by (\ref{j0})}\\
&=-\frac{1}{3}(2(-1)^n\sum_{j=1}^n(-2)^{j-1}-(-1)^n\sum_{j=1}^n1)\\
&=-\frac{1}{3}(2(-1)^n\frac{1-(-2)^n}{3}-(-1)^nn)\\
&=-\frac{1}{3}(-\frac{2}{3}(2^n-(-1)^n)+n(J_n-J_{n-1})) &&\text{by (\ref{j2})}\\
&=-\frac{1}{3}(-2J_n+n(J_n-J_{n-1})) &&\text{by (\ref{j0})}\\
&=\frac{1}{3}(nJ_{n-1}-(n-2)J_n).
\end{align*}

\end{proof}

%%%%%%%%%%%%%%%%%%%%%%%%%%%%%%%%%%%%%%%%%%%%%%%%%%
\section{Proof of main theorem}
The laplacian matrix of graph $\cay(\ZZ_N,\{+1,+2\})$ is a matrix $L$ whose entries $L(i,j)$ are given by
\[
L(i,j)=\begin{cases}
\deg_{\cay(\ZZ_N,\{+1,+2\})}(i) &\text{if $i=j$,}\\
-1 &\text{if $i\neq j$ and there is an arc from $i$ to $j$,}\\
0 &\text{otherwise.}
\end{cases}
\]
Let $L^{\prime}$ be the matrix obtained from $L$ by deleting the last row and column. 
Let $\vec{h}$ be the column vector whose $i$-th entry is $h(\cay(\ZZ_N,\{+1,+2\});0,i)$ and 
let $\vec{1}$ be a column vector of proper dimension whose entries are all $1$.
In the case of a random walk on $\cay(\ZZ_N,\{+1,+2\})$, a random walker moves with probability $\frac{1}{2}$ to an arbitrary vertex adjacent to the vertex where the walker is.
Hence we have that $\forall l \in\mathbb{Z}_N\setminus\{0\}$, 
$h(\cay(\ZZ_N,\{+1,+2\});0,l)=\frac{1}{2}((1+h(\cay(\ZZ_N,\{+1,+2\});1,l)+(1+h(\cay(\ZZ_N,\{+1,+2\});2,l)))$,
and hence $\forall l \in\mathbb{Z}_N\setminus\{0\}, 2h(\cay(\ZZ_N,\{+1,+2\});0,l)-h(\cay(\ZZ_N,\{+1,+2\});0,l-1)-h(\cay(\ZZ_N,\{+1,+2\});0,l-2)=2$. 
So we have ${}^tL' \vec{h}=2\vec{1}$.
Note that $h(\cay(\ZZ_N,\{+1,+2\});0,l)=h(\cay(\ZZ_N,\{+1,+2\});k,k+l)=h(\cay(\ZZ_N,\{+1,+2\});k+N-l,k)$ for all $k,l \in \mathbb{Z}_{N}$.
Let $L'(i, j)$ denote the $(i, j)$-th entry of $L'$. 
Let $H_N$ be the $(N-1)\times(N-1)$-matrix whose $(i,j)$-th entry is 
${}^tL'(i, j)$.
% if $j=\frac{N}{2} (N=2n)$, and $L'(i, j)+L'(i, N-j)$ otherwise. 
Then we have $H_N\vec{h}=2\vec{1}$.
For our problem it is sufficient to solve this matrix equation.
%As a technical issue, in the case of $N=2n$, our matrix $H_{2n}$ is not symmetric.
%However, if we multiply the last row of this matrix by $\frac{1}{2}$, then the resultant matrix
%$H'_{2n}$ will be symmetric. Hence, only in the case of $N=2n$, we redefine
%$H_{2n}:=H'_{2n}$ and change our matrix equation to $H_{2n} \vec{x}=$${}^{t}(4,\ldots,4,2)$.

%%%%%%%%%%%%%%%%%%%%%%%%%%%%%%%%%%%%%%%%%%%%%%%%%%%%%%%%%%%%%%%%%%%%%%%%%%%
Then we have the following:
\begin{thm}\label{HNinverse}
The entries of $H_N^{-1}$ is as follows:
\[
H_N^{-1}(i,j)=\begin{cases}
\frac{J_iJ_{N-i-1}}{J_N} &(j=i+1)\\
\frac{J_{i-j+1}(J_{N-i+j-2}+J_{N-i+j-1})-(-1)^{i+j}J_{j-1}J_{N-i-1}}{J_N} &(j<i+1,\;(i,j)\neq(N-1,1))\\
\frac{J_iJ_{N-j}(J_{j-i-1}+J_{j-i})}{J_N} &(j>i+1)\\
\frac{J_{N-1}}{J_N} &((i,j)=(N-1,1))
\end{cases}
\]
\end{thm}

\begin{proof}
Let $S_N$ denote the matrix as follows. 
\[
S_N(i,j)=
\begin{cases}
\frac{J_iJ_{N-i-1}}{J_N} &(j=i+1)\\
\frac{J_{i-j+1}(J_{N-i+j-2}+J_{N-i+j-1})-(-1)^{i+j}J_{j-1}J_{N-i-1}}{J_N} &(j<i+1,\;(i,j)\neq(N-1,1))\\
\frac{J_iJ_{N-j}(J_{j-i-1}+J_{j-i})}{J_N} &(j>i+1)\\
\frac{J_{N-1}}{J_N} &((i,j)=(N-1,1))
\end{cases}
\]

We will show $H_NS_N$ equals to the identity matrix.
By simple matrix-computation, the entries of $H_NS_N$ can be calculated as follows.

In the case that $(i,j)=(1,1)$, we have
\begin{align*}
  H_NS_N(i,j)&=\frac{1}{J_N}(2J_1(J_{N-2}+J_{N-1})-2J_0J_{N-2}-J_{N-1})\\
  &=\frac{1}{J_N}(2(J_{N-2}+J_{N-1})-J_{N-1})\\
  &=\frac{1}{J_N}(2J_{N-2}+J_{N-1})\\
  &=1.
\end{align*}

In the case that $(i,j)=(1,2)$, we have
\begin{align*}
  H_NS_N(i,j)&=\frac{1}{J_N}(2J_1J_{N-2}-J_{N-2}(J_1+J_2)+(-1)^{N+1}J_1J_0)\\
  &=\frac{1}{J_N}(2J_{N-2}-2J_{N-2})\\
  &=0.
\end{align*}

In the case that $i=1,\;j\notin\{1,2\}$, we have
\begin{align*}
  H_NS_N(i,j)&=\frac{1}{J_N}(2J_1J_{N-j}(J_{j-2}+J_{j-1})-J_{N-j}(J_{j-1}+J_j)+(-1)^{N+j-1}J_{j-1}J_0)\\
  &=\frac{1}{J_N}(2J_{N-j}(J_{j-2}+J_{j-1})-J_{N-j}(J_{j-1}+J_j))\\
  &=\frac{J_{N-j}}{J_N}(2J_{j-2}+J_{j-1}-J_j)\\
  &=0.
\end{align*}

In the case that $(i,j)=(2,1)$, we have
\begin{align*}
  H_NS_N(i,j)&=\frac{1}{J_N}(-J_1(J_{N-2}+J_{N-1})+J_0J_{N-2}+2J_2(J_{N-3}+J_{N-2})+2J_0J_{N-3})\\
  &=\frac{1}{J_N}(-J_{N-2}-J_{N-1}+2(J_{N-3}+J_{N-2}))\\
  &=\frac{1}{J_N}(2J_{N-3}+J_{N-2}-J_{N-1})\\
 &=0.
\end{align*}

In the case that $(i,j)=(2,2)$, we have
\begin{align*}
  H_NS_N(i,j)&=\frac{1}{J_N}(-J_1J_{N-2}+2J_1(J_{N-2}+J_{N-1})-2J_1J_{N-3})\\
  &=\frac{1}{J_N}(-J_{N-2}+2(J_{N-2}+J_{N-1})-2J_{N-3})\\
  &=\frac{1}{J_N}(-2J_{N-3}+J_{N-2}+2J_{N-1})\\
   &=\frac{1}{J_N}(2J_{N-2}+J_{N-1})\\
  &=1.
\end{align*}

In the case that $(i,j)=(2,3)$, we have
\begin{align*}
  H_NS_N(i,j)&=\frac{1}{J_N}(-J_1J_{N-3}(J_1+J_2)+2J_2J_{N-3})\\
  &=\frac{1}{J_N}(-2J_{N-3}+2J_{N-3})\\
  &=0.
\end{align*}

In the case that $i=2,\;j\notin\{1,2,3\}$, we have
\begin{align*}
  H_NS_N(i,j)&=\frac{1}{J_N}(-J_1J_{N-j}(J_{j-2}+J_{j-1})+2J_2J_{N-j}(J_{j-3}+J_{j-2}))\\
  &=\frac{1}{J_N}(-J_{N-j}(J_{j-2}+J_{j-1})+2J_{N-j}(J_{j-3}+J_{j-2}))\\
  &=\frac{J_{N-j}}{J_N}(-J_{j-1}+J_{j-2}+2J_{j-3})\\
  &=0.
\end{align*}

In the case that $i>j+1,i\notin\{1,2\},(i,j)\neq(N-1,1)$, we have
\begin{align*}
  H_NS_N(i,j)&=\frac{1}{J_N}(-J_{i-j-1}(J_{N-i+j}+J_{N-i+j+1})+(-1)^{i+j-4}J_{j-1}J_{N-i+1}\\
  &~-J_{i-j}(J_{N-i+j-1}+J_{N-i+j})+(-1)^{i+j-3}J_{j-1}J_{N-i}\\
  &~+2J_{i-j+1}(J_{N-i+j-2}+J_{N-i+j-1})-2(-1)^{i+j-2}J_{j-1}J_{N-i-1})\\
  &=\frac{1}{J_N}(-2J_{i-j-1}(J_{N-i+j-1}+J_{N-i+j})-J_{i-j}(J_{N-i+j-1}+J_{N-i+j})\\
  &~+J_{i-j+1}(J_{N-i+j-1}+J_{N-i+j})\\
  &~-(-1)^{i+j}J_{j-1}(2J_{N-i-1}+J_{N-i}-J_{N-i+1}))\\
  &=\frac{1}{J_N}(-(2J_{i-j-1}+J_{i-j}-J_{i-j+1})(J_{N-i+j-1}+J_{N-i+j}))\\
  &=0.
\end{align*}

In the case that $i=j+1,i\notin\{1,2\}$, we have
\begin{align*}
   H_NS_N(i,j)&=\frac{1}{J_N}(-J_{i-2}J_{N-i+1}-J_{i-j}(J_{N-i+j-1}+J_{N-i+j})+(-1)^{i+j-1}J_{j-1}J_{N-i}\\
   &~+2J_{i-j+1}(J_{N-i+j-2}+J_{N-i+j-1})-2(-1)^{i+j}J_{j-1}J_{N-i-1})\\
   &=\frac{1}{J_N}(-J_{j-1}J_{N-j}-J_1(J_{N-2}+J_{N-1})+J_{j-1}J_{N-j-1}\\
   &~+2J_2(J_{N-3}+J_{N-2})+2J_{j-1}J_{N-j-2})\\
   &=\frac{1}{J_N}(-J_{j-1}J_{N-j}-(J_{N-2}+J_{N-1})+J_{j-1}J_{N-j-1}\\
   &~+2(J_{N-3}+J_{N-2})+2J_{j-1}J_{N-j-2})\\
   &=\frac{1}{J_N}(J_{j-1}(2J_{N-j-2}+J_{N-j-1}-J_{N-j})+(2J_{N-3}+J_{N-2}-J_{N-1}))\\
   &=0.
\end{align*}

In the case that $i=j,i\notin\{1,2\}$, we have
\begin{align*}
  H_NS_N(i,j)&=\frac{1}{J_N}(-J_{i-2}J_{N-j}(J_{j-i+1}+J_{j-i+2})-J_{i-1}J_{N-i}\\
  &~+2J_{i-j+1}(J_{N-i+j-2}+J_{N-i+j-1})-2(-1)^{i+j}J_{j-1}J_{N-i-1})\\
  &=\frac{1}{J_N}(-J_{j-2}J_{N-j}(J_1+J_2)-J_{j-1}J_{N-j}\\
  &~+2J_1(J_{N-2}+J_{N-1})-2J_{j-1}J_{N-j-1})\\
  &=\frac{1}{J_N}(-2J_{j-2}J_{N-j}-J_{j-1}J_{N-j}+2(J_{N-2}+J_{N-1})\\
  &~-2J_{j-1}J_{N-j-1})\\
  &=\frac{1}{J_N}(-(2J_{j-2}+J_{j-1})J_{N-j}+(J_{N-1}+J_N)-2J_{j-1}J_{N-j-1})\\
  &=\frac{1}{J_N}(-J_jJ_{N-j}-2J_{j-1}J_{N-j-1}+J_{N-1}+J_N)\\
  &=\frac{1}{J_N}(-J_{N-1}+J_{N-1}+J_N) &&\text{by (\ref{lj1})}\\
  &=1.
\end{align*}

In the case that $i=j-1,i\notin\{1,2\}$, we have
\begin{align*}
  H_NS_N(i,j)&=\frac{1}{J_N}(-J_{i-2}J_{N-j}(J_{j-i+1}+J_{j-i+2})-J_{i-1}J_{N-j}(J_{j-i}+J_{j-i+1})\\
  &~+2J_iJ_{N-i-1})\\
  &=\frac{1}{J_N}(-J_{j-3}J_{N-j}(J_2+J_3)-J_{j-2}J_{N-j}(J_1+J_2)+2J_{j-1}J_{N-j})\\
  &=\frac{1}{J_N}(-4J_{j-3}J_{N-j}-2J_{j-2}J_{N-j}+2J_{j-1}J_{N-j})\\
  &=\frac{2J_{N-j}}{J_N}(-2J_{j-3}-J_{j-2}+J_{j-1})\\
  &=0.
\end{align*}

In the case that $i<j-1,i\notin\{1,2\}$, we have
\begin{align*}
  H_NS_N(i,j)&=\frac{1}{J_N}(-J_{i-2}J_{N-j}(J_{j-i+1}+J_{j-i+2})-J_{i-1}J_{N-j}(J_{j-i}+J_{j-i+1})\\
  &~+2J_iJ_{N-j}(J_{j-i-1}+J_{j-i}))\\
  &=\frac{J_{N-j}}{J_N}(-2J_{i-2}(J_{j-i}+J_{j-i+1})-J_{i-1}(J_{j-i}+J_{j-i+1})\\
  &~+J_i(J_{j-i}+J_{j-i+1}))\\
  &=\frac{J_{N-j}}{J_N}(-2J_{i-2}-J_{i-1}+J_i)(J_{j-i-1}+J_{j-i})\\
  &=0.
\end{align*}

In the case that $(i,j)=(N-1,1)$, we have
\begin{align*}
  H_NS_N(i,j)&=\frac{1}{J_N}(-J_{N-3}(J_2+J_3)+(-1)^{N-2}J_0J_2-J_{N-2}(J_1+J_2)\\
  &~+(-1)^{N-1}J_0J_1+2J_{N-1})\\
  &=\frac{1}{J_N}(-4J_{N-3}-2J_{N-2}+2J_{N-1})\\
  &=0.
\end{align*}

Therefore, we have $S_N=H_N^{-1}$.
\end{proof}

Combining  Theorem~\ref{HNinverse} and the matrix equation
\[
\vec{h}=
\begin{bmatrix}
h(\cay(\ZZ_N,\{+1,+2\});0,1)\\
h(\cay(\ZZ_N,\{+1,+2\});0,2)\\
\vdots \\
h(\cay(\ZZ_N,\{+1,+2\});0,N-2) \\
h(\cay(\ZZ_N,\{+1,+2\});0,N-1)\\
\end{bmatrix}
=2H_N^{-1}
\begin{bmatrix}
1\\
1\\
\vdots\\
1\\
1
\end{bmatrix},
\]
we obtain the following. 

\begin{thm}
For $1\leq l\leq N-1$, the exact formula for the HT's of simple random walks on $\cay(\ZZ_N,\{+1,+2\})$ is,
\begin{align*}
&h(\cay(\ZZ_N,\{+1,+2\});0,l)=\frac{1}{3J_N}(2J_{l-1}(3lJ_{N-l-1}+2lJ_{N-l})\\
&~+J_l((N+l+3)J_{N-l-1}+(N+3l+1)J_{N-l})).
\end{align*}
\end{thm}

\begin{proof}
For $1\leq l\leq N-2$,
\begin{align*}
&h(\cay(\ZZ_N,\{+1,+2\});0,l)\\
&=2\sum_{j=1}^{N-1}H_N^{-1}(l,j)\\
&=\frac{2}{J_N}(\sum_{j=1}^l(J_{l-j+1}(J_{N-l+j-2}+J_{N-l+j-1})-(-1)^{l+j}J_{j-1}J_{N-l-1})\\
&~+J_lJ_{N-l-1}+\sum_{j=l+2}^{N-1}J_lJ_{N-j}(J_{j-l-1}+J_{j-l}))\\
&=\frac{2}{J_N}(\sum_{j=1}^l2^{N-l+j-2}J_{l-j+1}-(-1)^lJ_{N-l-1}\sum_{j=1}^l(-1)^jJ_{j-1}\\
&~+J_lJ_{N-l-1}+J_l\sum_{j=l+2}^{N-1}2^{j-l-1}J_{N-j}) &&\text{by (\ref{j1})}\\
&=\frac{2}{J_N}(2^{N-l-2}\sum_{j=1}^{l+1}2^jJ_{l-j+1}+J_{N-l-1}\sum_{j=1}^l(-1)^{l+j}J_j\\
&~+J_l\sum_{j=1}^{N-l-1}2^jJ_{N-l-j-1})\\
&=\frac{2}{3J_N}(2^{N-l-1}((l+1)J_l+lJ_{l+1})+J_{N-l-1}(lJ_{l-1}-(l-2)J_l)\\
&~+J_l((N-l-1)J_{N-l-2}+(N-l-2)J_{N-l-1})) &&\text{by (\ref{lj2}),(\ref{lj3})}\\
&=\frac{2}{3J_N}(2(J_{N-l-2}+J_{N-l-1})((2l+1)J_l+2lJ_{l-1})\\
&~+J_{N-l-1}(lJ_{l-1}-(l-2)J_l)+J_l((N-l-1)J_{N-l-2}\\
&~+(N-l-2)J_{N-l-1})) &&\text{by (\ref{j1})}
\intertext{}
&=\frac{2}{3J_N}(J_{l-1}(4lJ_{N-l-2}+5lJ_{N-l-1})\\
&~+J_l((N+3l+1)J_{N-l-2}+(N+2l+2)J_{N-l-1}))\\
&=\frac{1}{3J_N}(2J_{l-1}(3lJ_{N-l-1}+2lJ_{N-l})\\
&~+J_l((N+l+3)J_{N-l-1}+(N+3l+1)J_{N-l})).
\end{align*}
In the case of $l=N-1$, we have
\begin{align*}
&h(\cay(\ZZ_N,\{+1,+2\});0,N-1)\\
&=2\sum_{j=1}^{N-1}H_N^{-1}(N-1,j)\\
&=\frac{2}{J_N}(J_{N-1}+\sum_{j=2}^{N-1}J_{N-j}(J_{j-1}+J_j))\\
&=\frac{2}{J_N}(J_{N-1}+\sum_{j=2}^{N-1}2^{j-1}J_{N-j}) &&\text{by (\ref{j1})}\\
&=\frac{1}{J_N}\sum_{j=1}^N2^jJ_{N-j}\\
&=\frac{2}{3J_N}(NJ_{N-1}+(N-1)J_N)  &&\text{by (\ref{lj2})}\\
&=\frac{2}{3J_N}(2(N-1)J_{N-2}+(2N-1)J_{N-1})\\
&=\frac{1}{3J_N}(4(N-1)J_{N-2}J_1+2(2N-1)J_{N-1}J_1)\\
&=\frac{1}{3J_N}(2J_{N-2}(3(N-1)J_0+2(N-1)J_1)\\
&~+J_{N-1}((2N+2)J_0+(4N-2)J_1))\\
&=\frac{1}{3J_N}(2J_{N-2}(3(N-1)J_0+2(N-1)J_1)\\
&~+J_{N-1}((N+(N-1)+3)J_0+(N+3(N-1)+1)J_1)).
\end{align*}

\end{proof}

%\begin{rem}
%We also obtain the counterpart of Theorem \ref{decomposition} for the case of $C^3_N$ \cite{C3N},
%which convinced us that our method is always efficient for the case of general 
%$C^m_N$, although their calculations get much complicated as the parameter $m$ get larger. 
%\end{rem}

\section*{Acknowledgements}

The authors thank Tsuyoshi Miezaki and Hiroshi Suzuki for their helpful discussions 
and comments to this research.

This work was supported by JSPS Grant-in-Aid for JSPS Fellows (23KJ2020) and Waseda Research Institute for Science and Engineering, Grant-in-Aid for Young Scientists(Early Bird).

%The authors would also like to thank the anonymous
%reviewers for their beneficial comments on an earlier version of the manuscript.

%%%%%%%%%%%%%%%%%  References  %%%%%%%%%%%%%%%%%%%%%%%%


\begin{thebibliography}{99}
\bibitem{Horadam} A. F. Horadam. : Jacobsthal representation numbers,\textit{The Fibonacci Quarterly}, 34.1 (1996) 68-74.
 
\bibitem{Chair}  N. Chair. : The Effective Resistance of the {$N$}-Cycle Graph with Four Nearest Neighbors, \textit{Journal of Statistical Physics}, \textbf{154} (2014) 1177-1190.

\bibitem{HT} Y. Doi, N. Konno, T. Nakamigawa, T. Sakuma, E. Segawa, H. Shinohara, S. Tamura, Y. Tanaka, and K. Toyota. :  On the average hitting times of the squares of cycles, \textit{Discrete Applied Mathematics}, \textbf{313} (2022) 18-28.

%\bibitem{Simon} S. Jacob. : \textit{Ein new und wohlgegr\"{u}ndt Rechenbuch, auf den Linien und Ziffern, Sampt der Welschen Practica ...}, Frankfurt am Meyn, Christian Egenolff, 1571.

%\bibitem{Kirchhoff}  G. Kirchhoff. : Uber die Aufl\"{o}sung der Gleichungen, auf welche man bei der Untersuchung der linearen Verteilung galvanischer Str\"{o}me gef\"{u}hrt wird, \textit{Annalen Physik und Chemie}, \textbf{72} (1847) 497-508

%\bibitem{Kleitman}  D. J. Kleitman, B. Golden. : Counting Trees in a Certain Class of Graphs, \textit{The American Mathematical Monthly}, \textbf{82} No.1 (1975) 40-44

%\bibitem{Nash}  C. S. J. A. Nash-Williams. : Random walk and electric currents in networks. \textit{Proceedings of the Cambridge Philosophical Society. Mathematical and Physical Sciences}, \textbf{55} (1959) 181-194.

%\bibitem{Thomas} T. Koshy. : \textit{Fibonacci and Lucas Numbers with Applications, Volume 1}, Pure and Applied Mathematics: A Wiley Series of Texts, Monographs and Tracts, 2018.

%\bibitem{Koshy} T. Koshy. : \textit{Fibonacci and Lucas Numbers with Applications, Volume 2}, Pure and Applied Mathematics: A Wiley Series of Texts, Monographs and Tracts, 2019.

\bibitem{Wu}  F. Y. Wu. : Theory of resistor networks: the two-point resistance, \textit{Journal of Physics A: Mathematical and General}, \textbf{37}(26) (2004) 6653-6673. 

\end{thebibliography}
\end{document}